\documentclass[11pt]{amsart}
\usepackage{amsmath,amssymb,latexsym}
\usepackage[left=2.6cm,right=2.6cm,top=2.7cm,bottom=2.7cm]{geometry}
\usepackage[english]{babel}
\usepackage{esint}
\newtheorem{theorem}{Theorem}[section]
\newtheorem{corollary}[theorem]{Corollary}
\newtheorem{lemma}[theorem]{Lemma}

\theoremstyle{definition}

\newtheorem{remark}[theorem]{Remark}

\numberwithin{equation}{section}

\renewcommand{\div}{\mbox{\rm div}\,}

\newcommand{\R}{{\mathbb R}}

\title[Global $L^\infty$-estimate in arbitrary domains]
{Global $L^\infty$-estimate for general quasilinear elliptic equations in arbitrary domains of $\mathbb{R}^N$}

\author[S. Carl]{Siegfried Carl}

\address[S. Carl]{Institut f\"ur Mathematik,  Martin-Luther-Universit\"at Halle-Wittenberg,
D-06099 Halle, Germany}
\email{\tt siegfried.carl@mathematik.uni-halle.de}

\author[H. Tehrani]{Hossein Tehrani}
\address[H. Tehrani]{Department of Mathematical Sciences,  University of Nevada Las Vegas Box 454020, USA}
\email{\tt tehranih@unlv.nevada.edu}

\keywords{Quasilinear Elliptic Equation, Beppo-Levi Space, Global $L^\infty$-Estimate, Wolff Potential, Decay property}

\subjclass[2010]{35B40, 35B45, 35J62}

\date{\today}

\begin{document}

\begin{abstract}
In this paper our main goal is to present a new global $L^\infty$-estimate for a general class of quasilinear elliptic equations of the form
$$
-\div \mathcal{A}(x,u,\nabla u)=\mathcal{B}(x,u,\nabla u)
$$
under minimal structure conditions on the functions $\mathcal{A}$ and $\mathcal{B}$, and in arbitrary domains of $\R^N$. The main focus and the novelty of the paper is to prove $L^\infty$-estimate of the form
$$
|u|_{\infty, \Omega}\le C \Phi(|u|_{\beta,\Omega})
$$
where the constant $C$ encodes  the contribution of the data, and $\Phi: \R^+\to \R^+$ is a data independent, continuous, and nondecreasing function with $\lim_{s\to 0^+}\Phi(s)=0$.
\end{abstract}

\maketitle


 \section{Introduction}\label{S1}

In this note our goal is to present a new $L^{\infty}$-estimate for solutions to a general class of quasilinear elliptic equations
in divergence form:
\begin{equation*}
\label{(P1)}
(P1)\quad\quad\,  -\div \mathcal{A}(x,u,\nabla u)=\mathcal{B}(x,u, \nabla u)
\end{equation*}
where $\mathcal{A}: \Omega\times \mathbb{R} \times \mathbb{R}^N\rightarrow \mathbb{R}^N$, and
 $\mathcal{B}:\Omega\times \mathbb{R} \times \mathbb{R}^N\rightarrow \mathbb{R}$ are given Carath\'eodory functions with
$\mathcal{A}(\cdot , \zeta, \xi),\, \mathcal{B}(\cdot , \zeta, \xi)$ Lebesgue measurable, and $\mathcal{A}(x, \cdot , \cdot),\, \mathcal{B}(x, \cdot , \cdot)$ are  assumed to be continuous. Also throughout we assume the following structural conditions:
\begin{eqnarray*}
(s1)\qquad\quad \, \mathcal{A}(x , \zeta, \xi)\xi & \geq & \lambda |\xi|^p- a_1(x)|\zeta|^p - a_2(x) \\
(s2)\qquad\quad \, |\mathcal{B}(x , \zeta, \xi)| & \leq & b_1(x) |\xi|^{p-1} +  b_2 (x)|\zeta|^{p-1}+ b_3(x)
\end{eqnarray*}
where $1< p< N\, $, $\lambda $ a fixed positive constant and the remaining nonnegative functions $a_i$ and $b_j$ in (s1)--(s2) are assumed to be in suitable integral spaces specified below. We consider weak $ D^{1,p}(\Omega)$ solutions of $(P1)$ where $\Omega$ is an arbitrary (not necessarily bounded) domain of $\mathbb{R}^N$, and $ X:= D^{1,p}(\Omega)$  is  the Beppo-Levi space (homogeneous Sobolev space) which is the completion of $C_c^{\infty}(\Omega)$ under the norm $\|u\|_X=(\int_{\Omega} |\nabla u|^p)^{1/p}$. Hence,  $u\in D^{1,p}(\Omega)$ is a solution of $(P1)$ if the following holds:
\begin{equation}\label{1.1}
\int_{\Omega} \mathcal{A}(x,u,\nabla u)\nabla \phi\, dx =\int_{\Omega}\mathcal{B}(x,u,\nabla u)\phi \, dx  \quad\forall \phi\in D^{1,p}(\Omega).
\end{equation}
To motivate our result, let $\Omega=\R^N$. Apparently $X= D^{1,p}(\R^N)\subset W_{loc}^{1,p}(\R^N)$. We note that by piecing together the classical local $L^{\infty}$-estimates for $W_{loc}^{1,p}(\R^N) $ solutions of $(P1)$, namely;
\begin{equation}\label{1.2}
|u|_{\infty, B(x_0,r)} \leq C( |u|_{p^{*}, B(x_0,2r)}+K(r))
\end{equation}
one can obtain the following global bound for a  $D^{1,p}(\mathbb{R}^N)$ solution $u$ of $(P1)$ in $\mathbb{R}^N$:
\begin{equation}\label{1.21}
 |u|_{\infty, \mathbb{R}^N} \leq C( |u|_{p^*, \mathbb{R}^N}+K)
\end{equation}
where $p^*=\frac{Np}{N-p}$ is the critical Sobolev exponent and the constants $C$, and in particular, $K$ depend on suitable integral norms of the structural data, i.e., the functions $a_1 $ through $b_3 $ (see \cite{Serrin}). In general such global $L^{\infty}$-estimates provide useful bounds for solutions of $(P1)$, which play a prominent role and are the basis for investigating  qualitative properties of solutions such as regularity  (see e.g. \cite{DiB83, Lie88, Tolksdorf}) as well as Harnack type inequalities for solutions, see \cite{Tru67}. On the other hand in a recent work on some quasilinear elliptic equations in the whole space $\mathbb{R}^N$, we were naturally led to consider a control on the $L^{\infty}$ norm that depends only on $L^{p^*} $ norm of the solution itself so that, in particular, the smallness of integral norm will directly result in pointwise smallness of the solution. As one can easily observe, the existence of the data dependent term $K$ on the right hand side of (\ref{1.21}), renders this global estimate of limited use in this context. Therefore, we are interested in obtaining  $L^{\infty}$-estimates for solutions of $(P1)$ of the form:
\begin{equation}\label{goal}
|u|_{\infty, \Omega} \leq C \Phi (|u|_{\beta, \Omega})
\end{equation}
where $\Phi: \mathbb{R}^+\to \mathbb{R}^+$ is a {\it data independent} {\bf continuous, nondecreasing  function}  satisfying $\lim_{s\rightarrow 0^+}\Phi(s)=0$, and where  the constant $C$ encodes the contribution of the data functions $a_1 $ through $b_3 $ in the structure conditions (s1)--(s2), which are of different integrability.
In particular, the above estimates (\ref{goal}) holds true for $\beta=p^*$ such that (\ref{goal}) implies
\begin{equation}\label{goal-1}
|u|_{\infty, \Omega}\le C \Phi(\|u\|_X),
\end{equation}
which yields, in particular, $|u|_{\infty}\to 0$ as $\|u\|_X\to 0$.
With this result, one can simplify and replace some of the ad hoc arguments that were
previously used to control $L^\infty$-bounds of solutions in some previous works (see e.g. \cite{CCT-JDE,CT-JEPE}).

Although we have been unable to locate estimates of this type  and generality in the literature (not even for second order linear elliptic equations that corresponds to the case $p=2$ in our presentation here), we have nevertheless observed that the standard Moser iteration method that is generally employed in the proof of the classical local boundedness result (\ref{1.2}), can also be used in such a way as to result in a global estimate of the type that we proposed above. Therefore, the main goal and the novelty of this paper is to state and prove such a global   $L^{\infty}$-estimate of the form (\ref{goal}) for solutions of $(P1)$ in an {\bf arbitrary domain} $\Omega\subset \mathbb{R}^N$. In order to obtain our results, we make use of Moser's iteration method starting with an appropriately designed  truncated test function that is different from what is usually used in boundedness estimates based on  Nash-Moser theory. As for the Moser iteration technique in its standard form for some special elliptic equation in bounded domains we refer to e.g. \cite[Chap. 3]{Drabek-1997}.

Since the data functions $a_1 $ through $b_3 $ in the structure conditions (s1)--(s2) are assumed to belong to certain Lebesgue spaces, we will see that the following structure conditions are special cases of (s1)--(s2):
\begin{eqnarray*}
(\hat s1)\qquad\quad \, \mathcal{A}(x , \zeta, \xi)\xi & \geq & \lambda |\xi|^p- a_1 |\zeta|^{r} - a_2(x) \\
(\hat s2)\qquad\quad \, |\mathcal{B}(x , \zeta, \xi)| & \leq & b_1(x) |\xi|^{p-1} +  b_2 |\zeta|^{r-1}+ b_3(x),
\end{eqnarray*}
where $a_1$ and $b_2$ are positive constants, and $p< r < p^*$. In a number of a-priori boundedness results for quasilinear elliptic equations in bounded as well as unbounded domains the structural  data functions $a_1,\ b_1, b_2$ are assumed to be positive constants, and a-priori $L^\infty$ results are obtained under various additional assumptions on $\mathcal{A}$ and $\mathcal{B}$ (see e.g. \cite{Byun-Pala2013, Puel2004, Ma-Wi2019}), and for nonconstant data functions, see e.g. \cite{Liang-Wu1997}. Moreover, in those works, usually the $L^\infty$-bound is given by a constant $C$ that encodes the contribution of the data functions or is given in the form (\ref{1.21}).  The $L^\infty$-boundedness estimate (\ref{goal}) is different from those works  in that it only requires the structure conditions (s1)--(s2), which already allows us to control the $L^\infty$-bound by the norm of the underlying solution space $X$.

Our main result will be presented in Section \ref{S2} along with further implications that are derived from the main theorem. In Section \ref{S3} we study regularity, decay, and existence results of solutions for a quasilinear elliptic equation in $\R^N$ based on our main result.


\section{Main Result}\label{S2}

Before we present our result, first a few words on the notation. For an open set $V\subset \mathbb{R}^N$, the standard norms of the Lebesgue spaces  $L^r(V)$ are denoted by $|\cdot |_{r, V}$,  or whenever it is convenient and not confusing, by $|\cdot|_r$. An open ball with radius $r$ is denoted by $ B_r$, or $B(x_0,r)$ in case we want to emphasize the center, and $\chi_E$ denotes the characteristic function of the Lebesgue measurable set $E\subset \R^N$. Finally we use $C$, to denote a constant whose exact value is immaterial and may change from line to line. To indicate the dependence of the constant on the data, we write $C=C(a,b,\cdot,\cdot,\cdot)$ with the understanding that this dependence is increasing in its variables.
We are now ready to state our main result.
\begin{theorem}\label{T1}
Suppose $u\in D^{1,p}(\Omega)$ is a weak solution of equation $(P1)$:
$$
 -\div \mathcal{A}(x,u,\nabla u)=\mathcal{B}(x,u, \nabla u)
$$
where $\mathcal{A}$ and $\mathcal{B}$ satisfy the structural conditions $(s1)-(s2)$, with $ b_1\in L^{\infty}(\Omega)\, and $, $a_1\,, b_2\in L^{\frac{q}{p}}(\Omega)$ for some $q>N$. Furthermore assume $u\in  L^{\beta}(\Omega)$
for some $\beta >p$ and
$$
a_2\in  L^{\frac{\beta}{p}}(\Omega) \cap L^{\infty}(\Omega),\quad \quad  b_3\in  L^{\frac{\beta}{p-1}}(\Omega) \cap L^{\infty}(\Omega).
$$
Then there exists $\theta_0=\theta ( p,\beta,N), $ with $0< \theta_0 \leq 1$, such that
 \begin{equation}\label{1.3}
|u|_{\infty, \Omega} \leq C \max \{ |u|_{\beta, \Omega}, |u|^{\theta_0}_{\beta, \Omega}\}
\end{equation}
where $C=C\big( p,q,\beta,N,  |b_1|_{\infty,\Omega}, |a_1|_{\frac{q}{p}, \Omega}, |b_2|_{\frac{q}{p}, \Omega},||a_2||,||b_3||\big)$, with
$$||a_2||:=|a_2|_{\frac{\beta}{p}, \Omega}+|a_2|_{\infty,\Omega},\quad\quad ||b_3||:=|b_3|_{\frac{\beta}{p-1},\Omega}+|b_3|_{\infty,\Omega}.
$$
\end{theorem}
\begin{proof} As was mentioned above, the proof is based on Moser's iteration procedure and therefore the start of the proof
follows along typical steps to obtain a reverse H\"older inequality for the solution through standard estimates on properly chosen test functions. The point of departure in our proof is in the {\it nonlinear} nature of the reverse H\"older inequality that we obtain and the ensuing nontraditional iteration procedure.  Here as the details.

We consider the function $\phi=u\min{\{|u|^{\alpha-1}, L^p}\}$ for $\alpha =\beta-(p-1) >1$ and  $L >1$.
Clearly $\phi\in D^{1,p}(\Omega) $ is an admissible test function to be used in (\ref{1.1}) and
$$\nabla \phi =L^p \nabla u \chi_{\{|u|^{\alpha -1}> L^p\}} + \alpha |u|^{\alpha -1} \nabla u \chi_{\{|u|^{\alpha -1}\leq L^p\}}.
$$
Using the structure condition $(s1)$, and assuming $\lambda =1$ for simplicity, we have
\begin{eqnarray*}
\int_{\Omega}\mathcal{A}(x,u,\nabla u)\nabla \phi &\geq & \int_{\Omega}L^p|\nabla u|^p\chi_{\{|u|^{\alpha -1}> L^p\}}+
\alpha \int_{\Omega} |u|^{\alpha-1}|\nabla u|^p\chi_{\{|u|^{\alpha -1}\leq L^p\}}\\
& & - \int_{\Omega}a_1(x) L^p|u|^p\chi_{\{|u|^{\alpha -1}> L^p\}}-
\alpha \int_{\Omega}a_1(x) |u|^{\alpha +p-1}\chi_{\{|u|^{\alpha -1}\leq L^p\}} \\
& &   - \int_{\Omega}a_2(x) L^p\chi_{\{|u|^{\alpha -1}> L^p\}}
- \alpha \int_{\Omega}a_2(x) |u|^{\alpha -1}\chi_{\{|u|^{\alpha -1}\leq L^p\}}.
\end{eqnarray*}
Introducing
$$
\psi=u\min{\{|u|^{\frac{\alpha-1}{p}}, L}\}
$$
we observe that:
$$
|\nabla \psi|^p  =L^p |\nabla u|^p \chi_{\{|u|^{\alpha -1}> L^p\}} + \Big(\frac{\beta}{p}\Big)^p |u|^{\alpha -1} |\nabla u|^p \chi_{\{|u|^{\alpha -1}\leq L^p\}}.
$$
Next recalling that $\beta =\alpha +p-1$, and $1<\alpha <\beta$, we can rewrite the above inequality as
\begin{eqnarray}\label{1.4}
\int_{\Omega}|\nabla \psi|^p & \leq & \beta ^{p+1}\bigg\{\Big|\int_{\Omega}\mathcal{A}(x,u,\nabla u)\nabla \phi \Big|+\int_{\Omega} a_1(x)|\psi|^p +\int_{\Omega} a_2(x)|u|^{\alpha -1}\bigg\}.
\end{eqnarray}
Next we estimate the right hand side of (\ref{1.1}). Using the structure condition $(s2)$ we get:
\begin{eqnarray}\label{1.5}
\int_{\Omega}|\mathcal{B}(x,u,\nabla u)\phi | &\leq  & \int_{\Omega}b_1(x)|\nabla u|^{p-1}|u| \min{\{|u|^{\alpha-1}, L^p\}} + \int_{\Omega}b_2(x)| u|^{p}\min{\{|u|^{\alpha-1}, L^p\}}\\
& & \quad\quad \quad + \int_{\Omega}b_3(x)|u|^{\alpha} .\nonumber
\end{eqnarray}
Therefore combining (\ref{1.4}) and (\ref{1.5}), and using (\ref{1.1}) we obtain
\begin{eqnarray}\label{1.6}
\int_{\Omega}|\nabla \psi|^p & \leq &  \beta ^{p+1}\bigg\{ \int_{\Omega}b_1(x)|\nabla u|^{p-1}|u| \min{\{|u|^{\alpha-1},L^p\}}+ \int_{\Omega}(a_1(x)+b_2(x))|\psi |^{p}\\ \nonumber
& &  \quad\quad +\int_{\Omega} a_2(x)|u|^{\alpha -1}+ \int_{\Omega}b_3(x)|u|^{\alpha}\bigg\}.
\end{eqnarray}
Using H\"older and Young's inequalities, and taking note of the fact that $|\psi|^p\leq |u|^{\beta}$  we have
\begin{eqnarray}\label{1.7}
\int_{\Omega}b_1(x)|\nabla u|^{p-1}|u| \min{\{|u|^{\alpha-1},L^p\}} & =& \int_{\Omega}b_1(x)(L|\nabla u|)^{p-1}|Lu| \chi_{\{|u|^{\alpha -1}> L^p\}} \\   \nonumber
& & \quad  +  \int_{\Omega}b_1(x)|\nabla u|^{p-1}|u|^{\frac{(\alpha-1)(p-1)}{p}} |u|^{\frac{\beta}{p}}\chi_{\{|u|^{\alpha -1}\leq L^p\}}\\ \nonumber
& \leq & |b_1|_{\infty}\bigg\{ \varepsilon \Big( \int_{\Omega}|\nabla u|^{p}L^p \chi_{\{|u|^{\alpha -1}> L^p\}} \\ \nonumber
& &  + \int_{\Omega}|\nabla u|^{p}|u|^{(\alpha-1)}\chi_{\{|u|^{\alpha -1}< L^p\}}\Big)+C(\varepsilon)\int_{\Omega}|\psi|^p\bigg\}      \\ \nonumber
& \leq & |b_1|_{\infty}\bigg\{ \varepsilon \int_{\Omega}|\nabla \psi|^p + C(\varepsilon)\int_{\Omega}|u|^{\beta}\bigg\}.\nonumber
\end{eqnarray}
As for the second term on the right of (\ref{1.6}), noting that $q>N$, and again using $|\psi|^p\leq |u|^{\beta}$, we obtain
\begin{eqnarray}\label{1.8}
\int_{\Omega}(a_1(x)+b_2(x))|\psi|^p \, dx& \leq & |a_1+b_2|_{\frac{q}{p}}|\psi|^p_{\frac{pq}{q-p}}\\ \nonumber
& \leq &  \Big(|a_1|_{\frac{q}{p}}+|b_2|_{\frac{q}{p}}\Big)\Big(\tilde{\varepsilon}|\psi|^p_{p^{*}}+C(\tilde{\varepsilon}) |u|^{\beta}_{\beta}\Big).
\end{eqnarray}
Employing  Sobolev inequality in (\ref{1.6}) and choosing $\varepsilon$ and $\tilde{\varepsilon}$ appropriately small in (\ref{1.7}) and (\ref{1.8}) (in reference to (\ref{1.6}) ) we arrive at
\begin{eqnarray}\label{1.9}
\bigg(\int_{\Omega}|\psi |^{p^*}\bigg)^{\frac{p}{p^*}} & \leq & C \beta^{\sigma}\bigg\{ \int_{\Omega}|u|^{\beta}+\int_{\Omega} a_2(x)|u|^{\alpha -1}+ \int_{\Omega}b_3(x)|u|^{\alpha}\bigg\}
\end{eqnarray}
where $\sigma=\sigma (p,q,N)>1$ and $C=C \big( p,q,N, |b_1|_{\infty}, |a_1|_{\frac{q}{p}}, |b_2|_{\frac{q}{p}}\big)$.

As for the last two terms in the right hand side of (\ref{1.9}), we estimate as follows
$$ \int_{\Omega}a_2(x)|u|^{\alpha -1}\leq |a_2|_{\frac{\beta}{p}} \Big(\int_{\Omega}|u|^{\beta}\Big)^{\frac{\alpha -1}{\beta}}\leq ||a_2|| \Big(\int_{\Omega}|u|^{\beta}\Big)^{\frac{\alpha -1}{\beta}},
$$
and
$$ \int_{\Omega}b_3(x)|u|^{\alpha}\leq |b_3|_{\frac{\beta}{p-1}} \Big(\int_{\Omega}|u|^{\beta}\Big)^{\frac{\alpha}{\beta}}\leq ||b_3|| \Big(\int_{\Omega}|u|^{\beta}\Big)^{\frac{\alpha}{\beta}}.
$$
Therefore (\ref{1.9}) yields
\begin{equation*}
\bigg(\int_{\Omega} |u|^{\frac{\beta N}{N-p}} \chi_{\{|u|^{\alpha -1}< L^p\}}\bigg)^{\frac{N-p}{N}} \leq C \beta^{\sigma}\max\bigg\{ \int_{\Omega}|u|^{\beta}, \, \Big(\int_{\Omega}|u|^{\beta}\Big)^{\frac{\alpha -1}{\beta}}\bigg\}
\end{equation*}
where $C=C\big( p,q,N, |b_1|_{\infty}, |a_1|_{\frac{q}{p}}, |b_2|_{\frac{q}{p}}, ||a_2||, ||b_3||\big)$. Letting $L\rightarrow\infty$, we finally obtain
\begin{equation}\label{1.10}
|u|_{\beta\frac{N}{N-p}} \leq  (C\beta)^{\frac{\sigma}{\beta}}\max\Big\{ |u|_{\beta},
\, |u|_{\beta}^{1-\frac{p}{\beta}}\Big\}.
\end{equation}
Iterating this inequality, taking $\beta_0=\beta,\, \, \beta_k=\chi^k\beta\, $, with $\chi=\frac{N}{N-p}$, and keeping in mind that we may assume $C\beta \geq 1$,  we obtain
$$ |u|_{\chi^k\beta}\leq (C\chi\beta)^{\tilde{\sigma}}|u|_{\beta}^{\theta(k)},\quad\quad k\in \mathbb{Z}^{+}
$$
where $\tilde{\sigma}= \frac{\sigma}{\beta}(1+\sum_{m=1}^{\infty} \frac{m}{\chi^m})$ and
$$\theta(k)=\prod_{i=0}^{k-1} R(i),\quad\quad R(i)=1\, \,  \mbox{ or }\, \, R(i)=1-\frac{p}{\chi^i \beta}, \,\, \mbox{ for }\,
\, 0\leq i\leq k-1.
$$
Note that
$$0< \theta_0=\theta_0(p,\beta, N) :=\prod_{i=0}^{\infty}(1-\frac{p}{\chi^i \beta})\leq \theta(k)\leq 1
$$
as $\sum_{i=0}^{\infty} \frac{p}{\chi^i \beta}<\infty$.

Hence
\begin{equation}\label{1.11}
|u|_r\leq C\max\{ |u|_{\beta}^{\theta_0}, \, |u|_{\beta}\},\quad \forall r \mbox{ with } \beta\leq r<\infty,
\end{equation}
and with $C$ as in the statement of the theorem. This readily implies (\ref{1.3}), completing the proof of the theorem.
\end{proof}

Next, we are going to expand and clarify Theorem \ref{T1} by deriving new results based on it.

First, we apply Theorem \ref{T1} to get bounds for sub- and supersolutions of $(P1)$.
\begin{corollary}\label{C201}
Assume the hypotheses of Theorem \ref{T1} and let $u\in D^{1,p}(\Omega)$ be a subsolution of $(P1)$, that is, $u$ satisfies
$$
\int_{\Omega} \mathcal{A}(x,u,\nabla u)\nabla \phi\, dx \leq \int_{\Omega}\mathcal{B}(x,u,\nabla u)\phi \, dx  \quad\forall \phi\geq 0, \, \phi\in D^{1,p}(\Omega).
$$
Then the following bound for $u^+=\max\{u,0\}$ of the subsolution $u$ holds:
\begin{equation}\label{1.12}
\sup u^{+}\leq  C\max\{ |u^{+}|_{\beta}^{\theta_0}, \, |u^{+}|_{\beta}\}.
\end{equation}
\end{corollary}

\begin{proof}
In fact, starting with the test function $\phi=u^{+}\min\{|u^{+}|^{\alpha-1}, L^p\}$ the proof of Theorem \ref{T1}
goes through with minimal change, yielding the  global bound (\ref{1.12}).
\end{proof}
Similarly if $u\in D^{1,p}(\Omega)$ is a supersolution of $(P1)$, that is, $u$ satisfies
$$
\int_{\Omega} \mathcal{A}(x,u,\nabla u)\nabla \phi\, dx \geq \int_{\Omega}\mathcal{B}(x,u,\nabla u)\phi \, dx  \quad\forall \phi\geq 0, \, \phi\in D^{1,p}(\Omega),
$$
then we have the following result:
\begin{corollary}\label{C202}
Let $u$ be a supersolution of $(P1)$. Then under the hypotheses of Theorem \ref{T1} the following bound for $u^{-}=\max\{-u,0\}$ holds:
\begin{equation}\label{1.121}
\sup u^{-}\leq  C\max\{ |u^{-}|_{\beta}^{\theta_0}, \, |u^{-}|_{\beta}\}.
\end{equation}
\end{corollary}
\begin{proof}
This easily follows if one notes that if $u$ is a supersolution of $(P1)$ then $-u$ is a (weak) subsolution of the equation
\begin{equation*}
  -\div \mathcal{\tilde{A}}(x,u,\nabla u)=\mathcal{\tilde{B}}(x,u, \nabla u)
\end{equation*}
where the forms:
$$ \mathcal{\tilde{A}}(x , \zeta, \xi)=- \mathcal{A}(x , -\zeta, -\xi), \quad  \mathcal{\tilde{B}}(x , \zeta, \xi)=- \mathcal{B}(x , -\zeta, -\xi)
$$
clearly satisfy the structure conditions $(s1)-(s2)$ as before.
\end{proof}
In some applications it is helpful to have more precise information on the dependence of the constant $C$, on the right-hand side of the inequality (\ref{1.3}), on the norms $||a_2||$ and $||b_3||$. A more precise formulation can be obtained as follows.
\par\vspace{.1cm}
\noindent Given $t>0$ we define the forms:
$$
 \mathcal{\tilde{A}}(x , \zeta, \xi)=t^{p-1} \mathcal{A}(x , \frac{\zeta}{t}, \frac{\xi}{t}), \quad  \mathcal{\tilde{B}}(x , \zeta, \xi)=t^{p-1} \mathcal{B}(x , \frac{\zeta}{t}, \frac{\xi}{t}).
$$
Note that  $\mathcal{\tilde{A}}$ and $\mathcal{\tilde{B}}$ satisfy the structural conditions $(s1)-(s2)$ with the same data functions with the exception of $a_2$ and $b_3$ being replaced by $\tilde{a}_2(x)=t^pa_2(x)$ and $\tilde{b}_3(x)=t^{p-1}b_3(x)$, respectively. Observe that if $u$ solves $(P1)$ then $v=tu$ is a solution of
\begin{equation*}\label{p2}
(P2)\quad\quad \,  -\div \mathcal{\tilde{A}}(x,v,\nabla v)=\mathcal{\tilde{B}}(x,v, \nabla v)
\end{equation*}
Next we take $t=(\max \{1, ||a_2||^{\frac{1}{p}}, ||b_3||^{\frac{1}{p-1}} \})^{-1}$ and apply Theorem \ref{T1} to the solution $v=tu$ of the equation $(P2)$. Taking into account that now  $||\tilde{a}_2||\leq 1 \,$, $||\tilde{b}_3||\leq 1$ and $t\leq 1$, we get
\begin{eqnarray*}
|v|_{\infty} & \leq & C( p,q,N, |b_1|_{\infty}, |a_1|_{\frac{q}{p}}, |b_2|_{\frac{q}{p}})\max \{  |v|_{\beta}^{\theta_0}, \, |v|_{\beta}\}\\
& \leq &   C( p,q,N, |b_1|_{\infty}, |a_1|_{\frac{q}{p}}, |b_2|_{\frac{q}{p}}) t^{\theta_0}\max \{  |u|_{\beta}^{\theta_0}, \, |u|_{\beta}\}
\end{eqnarray*}
which, since $0<\theta_0 \leq 1$, yields the following refinement of (\ref{1.3}):
\begin{corollary}\label{C203}
Under the hypotheses of Theorem \ref{T1}, the following refined $L^\infty$-estimate holds:
\begin{equation}\label{1.13}
|u|_{\infty} \leq C( p,q,\beta, N,  |b_1|_{\infty}, |a_1|_{\frac{q}{p}}, |b_2|_{ \frac{q}{p}}) \max \{1, ||a_2||^{\frac{1}{p}}, ||b_3||^{\frac{1}{p-1}} \} \max \{ |u|_{\beta},|u|^{\theta_0}_{\beta}\}.
\end{equation}
\end{corollary}
The following corollary is related with the special structure conditions ($\hat s$1)--($\hat s$2).
\begin{corollary}\label{C204}
Let $u\in X$ be a solution of (P1) satisfying the structure conditions ($\hat s$1)--($\hat s$2) with $p< r < p^*$.
If $a_1$ and $b_2$ in ($\hat s$1)--($\hat s$2) are assumed to be positive constants, then $u$ satisfies the structure conditions (s1)--(s2).
\end{corollary}
\begin{proof}
Writing $a_1|u|^r=a_1|u|^{r-p}|u|^p$ and $b_2|u|^{r-1}=b_2|u|^{r-p}|u|^{p-1}$, we only need to show that $|u|^{r-p}\in L^{\frac{q}{p}}(\Omega)$. Since $(r-p)\frac{N}{p} < p^*$ (recall $p<N$), we may choose $q>N$ such that $(r-p)\frac{q}{p} = p^*$, which by recalling $X\hookrightarrow L^{p^*}(\Omega)$ completes the proof.
\end{proof}
\smallskip

\noindent{\bf Special case:} If in addition to the hypotheses of Theorem \ref{T1} we assume $a_1(x)=b_1(x)=b_2(x)=0$, an examination of the proof of the theorem indicates that inequality (\ref{1.9}) now reads:
\begin{equation*}
\Big(\int_{\Omega}|\psi |^{p^*}\Big)^{\frac{p}{p^*}} \leq C \beta^{\sigma}\Big\{ \int_{\Omega} a_2(x)|u|^{\alpha -1}+ \int_{\Omega}b_3(x)|u|^{\alpha}\Big\}
\end{equation*}
which in turn results in the following form of (\ref{1.10})
\begin{equation*}
|u|_{\beta\frac{N}{N-p}} \leq (C\beta)^{\frac{\sigma}{\beta}}\Big(||a_2||+ ||b_3||\Big)^{\frac{1}{\beta}} \max\Big\{ |u|_{\beta}^{1-\frac{p-1}{\beta}},
\, |u|_{\beta}^{1-\frac{p}{\beta}}\Big\}.
\end{equation*}
Iteration of this inequality and repeating the argument leading to Corollary 2.4, we finally get
\begin{equation}\label{1.14}
|u|_{\infty} \leq C( p,\beta, N) \Big(1+ ||a_2||^{\frac{1}{p}}+ ||b_3||^{\frac{1}{p-1}} \Big) \max\Big\{ |u|_{\beta}^{\theta_0},
\, |u|_{\beta}^{\theta_1}\Big\}
\end{equation}
with $\theta_0$ and $\theta_1$ depending only on $p,\beta$ and $ N$, and in particular, $0<\theta_0\leq \theta_1\leq 1-\frac{p-1}{\beta} <1$.
We note that if $\Omega$ has finite measure, by using $|u|_{\beta}^{\theta_0}\leq |u|_{\infty}^{\theta_0}|\Omega|^{\frac{\theta_0}{\beta}}$, and a corresponding inequality for $|u|_{\beta}^{\theta_1}$, and taking $\beta=p^{*}$ on the right hand side of (\ref{1.14}) we obtain an independent proof of the well known apriori estimate (see \cite[Theorem 3.12]{M-Z})
\begin{equation}\label{1.15}
|u|_{\infty} \leq C( p, N, |\Omega|) \Big(1+ ||a_2||^{\frac{1}{p}}+ ||b_3||^{\frac{1}{p-1}} \Big)^{\theta_2}
\end{equation}
for some $\theta_2=\theta_2(p,N)$.

\smallskip

\noindent{\bf Local $L^\infty$-estimate:} For $W_{loc}^{1,p}(\Omega)\cap L_{loc}^{\beta}(\Omega)$ solutions of $(P1)$ it is possible to derive a local version of (\ref{1.3}) if in addition to $(s1)-(s2)$, we further assume the following additional structural condition:
\begin{eqnarray*}
(s3)\qquad\quad \, |\mathcal{A}(x , \zeta, \xi)| & \leq & c_1(x)|\xi|^{p-1} +c_2(x)|\zeta|^{p-1} + c_3(x).
\end{eqnarray*}
In fact we follow the proof of Theorem \ref{T1} starting with the test function
$$
\tilde{\phi}=\eta^{p}\phi =\eta^p u \min{\{|u|^{\alpha-1}, L^p}\},
$$
where $\eta$ is an  appropriately chosen $C^{\infty}_c$ cutoff function and $\phi$ is as in the proof of the theorem. To begin with we may assume that $B_3:=B(x_0,3)\subset \Omega$
with  $\eta =1 $ on $B_{r_1}:=B(x_0,r_1)$ and $\eta=0$ on $B^c(x_0,r_2)$ with $1\leq r_1<r_2 < 3$. Note that $\nabla \tilde{\phi}=\eta^p\nabla \phi +p\eta^{p-1}\nabla \eta \phi$. Now using $(s3)$ to estimate the terms involving the additional term in $\nabla\tilde{\phi}$, that is,  $\eta^{p-1}\nabla \eta \phi$, standard  arguments similar to those used in the proof of Theorem \ref{T1} will result in the basic inequality:
$$
|u|_{\beta\frac{N}{N-p}, B_{r_1}} \leq  \Big(\frac{C\beta}{r_2-r_1}\Big)^{\frac{\sigma}{\beta}}\max\Big\{ \int_{B_{r_2}}|u|^{\beta},
\, \Big(\int_{B_{r_2}}|u|^{\beta}\Big)^{1-\frac{p-1}{\beta}}\Big \}
$$
for some $\sigma=\sigma(p,q,N)$.
Iteration of this inequality will lead to the following local form of our $L^{\infty}$ estimate:
$$ |u|_{\infty , B_1}\leq C \max \{|u|_{\beta, B_2}, |u|_{\beta, B_2}^{\theta_0}\}
$$
with $C$ depending on suitable integral norms of the data in the ball $B_2$. The following general case, then follows by the standard change of variable  $x\rightarrow rx$ (see \cite[Remark 3.2, p. 163]{M-Z}):
\begin{theorem}\label{T2}
Assume the hypothesis of Theorem \ref{T1}. In addition assume the additional structure condition $(s3)$. Suppose $B_{2r}\subset \Omega$, and for some $\beta>p$, the function $u\in W_{loc}^{1,p}(\Omega)\cap L_{loc}^{\beta}(\Omega)$ is a solution of $(P1)$, i.e.
\begin{equation*}
\int_{\Omega} \mathcal{A}(x,u,\nabla u)\nabla \phi\, dx =\int_{\Omega}\mathcal{B}(x,u,\nabla u)\phi \, dx,   \quad\forall \phi\in C^1_c(\Omega).
\end{equation*}
Further  assume $c_1 \in L^{\infty}(\Omega),\,  c_2\in L^{\frac{q}{p-1}}(\Omega)$, and
$ c_3\in L^{\frac{\beta}{p}}(\Omega)\cap L^{\infty}(\Omega)$.
Then there  exists $\theta_0=\theta ( p,\beta,N), $ with $0< \theta_0 \leq 1$, such that
\begin{equation}\label{1.16}
|u|_{\infty, B_{r}} \leq C \max \Big\{ \Big(\fint_{B_{2r}}|u|^{\beta}\Big)^{\frac{1}{\beta}},  \Big( \fint_{B_{2r}}|u|^{\beta}\Big)^{\frac{\theta_0}{\beta}}\Big\}
\end{equation}
where $C=C( p,q,\beta,N, r |b_1|_{\infty}, |c_1|_{\infty}, r^{p\delta_1}|a_1|_{ \frac{q}{p}}, r^{p\delta_1}|b_2|_{\frac{q}{p}},r^{(p-1)\delta_1} |c_2|_{\frac{q}{p-1}}, ||a_2||_r,||b_3||_r, ||c_3||_r)$, with
$$||a_2||_r=r^{p\delta_2}|a_2|_{\frac{\beta}{p}}+r^p|a_2|_{\infty}, \quad ||b_3||_r=r^{(p-1)\delta_2+1}|b_3|_{\frac{\beta}{p-1}}+r^p|b_3|_{\infty},
$$
and
$$
||c_3||_r=r^{(p-1)\delta_2}|c_3|_{\frac{\beta}{p-1}}+r^{p-1}|c_3|_{\infty}
$$
with $\delta_1=1-\frac{N}{q}, \quad \delta_2=1-\frac{N}{\beta}$, and all respective norms taken over $B_{2r}$.
\end{theorem}
\begin{remark}\label{R201}
The sign $\fint_{B_{2r}} $ in (\ref{1.16}) stands for $\frac{1}{|B_{2r}|}\int_{B_{2r}}$.
Finally note that in case $\Omega =\mathbb{R}^N$ one can easily obtain the global boundedness result of Theorem \ref{T1} from this local estimate. But it is important to note that the global estimate (\ref{1.3}) is obtained {\it for arbitrary domains $\Omega$ and without assuming the additional structural condition $(s3)$}.
\end{remark}
\begin{remark}\label{R202}
A $L^\infty$-estimate for a special class of quasilinear elliptic Dirichlet boundary value problem that is related with (\ref{goal-1}) of Section \ref{S1} has been obtained in \cite[Chap. 3]{Drabek-1997} for the following special problem, which is studied in the weighted Sobolev space $X=W^{1,p}_0(\nu, \Omega)$:
\begin{eqnarray*}
-\mathrm{div}(a(x,u)|\nabla u|^{p-2}\nabla u)&=&\mu b(x,u) |u|^{p-2}u+f(\mu, x,u) \ \ \mbox{in }\Omega,\\
u&=&0 \ \ \mbox{on }\partial\Omega,
\end{eqnarray*}
where $\Omega$ is a bounded domain, $\mu$ is some parameter, and $a,\,b,\,f$ are Carath\'eodory functions, which satisfy (taking for simplicity the weight $\nu(x)=1$ )
\begin{eqnarray*}
&& 0< \frac{1}{c_1}\le a(x,s)\le c_1 g(|s|), \quad\mbox{with } g: [0,\infty)\to [1,\infty) \ \mbox{being nondecrasing}\\
&& 0\le b(x,s)\le c_2(x)+c_3|s|^{r-p}\quad\mbox{with } p< r <p^*,\\
&& |f(\lambda, x,s)|\le c(\mu)(\sigma(x)+\varrho(x) |s|^{q-1}\quad\mbox{with } p< q <p^*
\end{eqnarray*}
One readily verifies that the structure condition ($\hat s$1)--($\hat s$2) are verified. In \cite[Lemma 3.14]{Drabek-1997} the following estimate was obtained
$$
|u|_\infty\le \hat c(\|u\|_X)
$$
for any weak solution, where $\hat c: \mathbb{R}^+\to \mathbb{R}^+$ is merely bounded on bounded sets. Even in this special case our general result (\ref{goal-1}) yields a more precise estimate.
\end{remark}


\section{A Quasilinear elliptic equation in $\R^N$: Regularity, Decay, and Existence}\label{S3}

In this section we are going to apply our main result of the preceding section to the following specific quasilinear elliptic problem in $\R^N$:
$$
u\in X=D^{1,p}(\R^N): -\div A(x,\nabla u)= a(x)g(x,u),
$$
where $A: \R^N\times \R^N\to \R^N$ is a Carath\'eodory function, that is, $x\mapsto A(x, \xi)$ is  measurable in $\R^N$ for all $\xi\in \R^N$, and $\xi\mapsto A(x,\xi)$ is continuous for a.e. $x\in \R^N$. The vector function $A$ is assumed to satisfy the following hypotheses:
\begin{itemize}
\item[(A1)] $|A(x,\xi)|\le \lambda |\xi|^{p-1}$;
\item[(A2)] $A(x,\xi)\xi\ge \nu |\xi|^p$;
\item[(A3)] $(A(x,\xi)-A(x,\hat\xi)(\xi-\hat\xi)> 0, \quad\forall \xi,\hat\xi\in \R^N, \ \xi\neq \hat\xi$.
\end{itemize}
Only for the sake of simplifying our presentation in this section, we are going to deal with the prototype $A(x,\xi)=|\xi|^{p-2}\xi$ with $\div \big(|\nabla u|^{p-2}\nabla u\big)=\Delta_p u$ the $p$-Laplacian, and  assume throughout $1 < p< N$. Thus in what follows we consider the elliptic equation
\begin{equation}\label{301}
u\in X=D^{1,p}(\R^N): -\Delta_p u= a(x)g(x,u),
\end{equation}
where the measurable function $a: \R^N\to\R$,  and the Carath\'eodory function  $g: \R^N\times \R\to \R$ are supposed to satisfy the following conditions:
\begin{itemize}
\item[(Ha)] $ |a(x)|\le c_a\frac{1}{1+|x|^{N+\alpha}},$ with some $c_a,\,\alpha >0$;
\item[(Hg)] $|g(x,s)|\le c_g\big(1+|s|^{r-1}\big),$ for a.e. $x\in \R^N$, $\forall s\in\R$ and some $c_g>0$, where $1\le r<p^*$.
\end{itemize}
Clearly $\mathcal{A}(x,\zeta,\xi)=|\xi|^{p-2}\xi$ fulfills the structure condition $(s1)$ with $\lambda=1$, $a_1=a_2=0$.
As for $\mathcal{B}(x,\zeta,\xi)=a(x)g(x,\zeta)$, we are going to show that the structure condition (s2) is satisfied as well.
We first note that the function
\begin{equation}\label{300}
w(x)=\frac{1}{1+|x|^{N+\alpha}}, \quad \alpha >0,
\end{equation}
is readily seen to belong to $L^\sigma(\R^N)$ for $\sigma\in[1,\infty]$. From $\mathcal{B}(x,\zeta,\xi)=a(x)g(x,\zeta)$ and (Hg) we get
$$
|\mathcal{B}(x,\zeta,\xi)|\le |a(x)|c_g+ |a(x)| c_g|\zeta|^{r-p}|\zeta|^{p-1}, \ \mbox{ with }  1\le r<p^*.
$$
Thus if $u\in X$ is a solution of (\ref{301}), then for the specific $\mathcal{B}(x,u,\nabla u)=a(x)g(x,u)$ we have an estimate related to (s2) with $b_1=0$, $b_2(x)=c_g|a(x)||u|^{r-p}$ and $ b_3(x)=c_g |a(x)|$. In order to verify that $\mathcal{B}(x,u,\nabla u)=a(x)g(x,u)$ satisfies the structure condition (s2), we need to show that $b_3\in L^{\frac{\beta}{p-1}}(\R^N)\cap L^\infty(\R^N)$ for some $\beta > p$, and $b_2\in L^{\frac{q}{p}}(\R^N)$ for some $q>N$. Since $w\in L^\sigma(\R^N)$ for $\sigma\in[1,\infty]$, it follows that $b_3\in L^\sigma(\R^N)$ for $\sigma\in[1,\infty]$. Let us verify that $b_2(x)=c_g|a(x)||u|^{r-p}$ belongs to $L^{\frac{q}{p}}(\R^N)$. To this end we only need to verify that $w|u|^{r-p}\in L^{\frac{q}{p}}(\R^N)$. Without loss of generality we may assume that $p< r <p^*$. Since $(r-p)\frac{N}{p}< p^*$ (note: $p<N$), we may choose $q> N$ such that $(r-p)\frac{q}{p}= p^*$. Recalling that $X=D^{1,p}(\R^N)$ is characterized through
$$
D^{1,p}(\R^N)=\{u\in L^{p^*}(\R^N): |\nabla u|\in L^p(\R^N)\},
$$
we obtain by taking into account $w\in L^\sigma(\R^N)$ for $\sigma\in[1,\infty]$  the estimate
$$
|w|u|^{r-p}|_{\frac{q}{p}}\le |w|_{\infty} |u|_{p^*}\le c |w|_{\infty} |u|_{X},
$$
for some positive constant $c$, which finally shows that $\mathcal{B}(x,u,\nabla u)=a(x)g(x,u)$ satisfies the structure condition (s2).

Hence we may apply the global $L^\infty$-estimate for solutions of (\ref{301}) provided by Theorem \ref{T1} with $\beta=p^*$ to get the following result.
\begin{corollary}\label{C301}
If $u\in X=D^{1,p}(\R^N)$ is a solution of (\ref{301}), then $u\in L^\infty(\R^N)$ and satisfies the estimate
\begin{equation}\label{302}
|u|_{\infty}\le C \max \big\{|u|_{p^*}, |u|_{p^*}^{\theta_0}\big\},
\end{equation}
where $\theta_0=\theta_0(p,N)$ with $0<\theta_0\le 1$, and the constant $C$ as in (\ref{1.3}) of Theorem \ref{T1}. Moreover, $u$ is $C^{1,\gamma}_{\mathrm{loc}}(\R^N)$-regular with $\gamma\in (0,1)$.
\end{corollary}
\begin{proof}
While the $L^\infty$-estimate is an immediate consequence of Theorem \ref{T1}, the regularity result is due to DiBenedetto, see \cite{DiB83}.
\end{proof}
From Corollary \ref{C301} it follows that the right-hand side of (\ref{301}) is bounded by
$$
|a(x)g(x,u(x))|\le C \frac{1}{1+|x|^{N+\alpha}}=C w(x),
$$
where the constant $C=C(c_a,c_g, |u|_{p^*})$. Let us consider the following quasilinear equation
\begin{equation}\label{303}
v\in X: -\Delta_p v= Cw(x),
\end{equation}
with $w$ given by (\ref{300}).
\begin{lemma}\label{L301}
Equation (\ref{303}) has a unique positive solution $v\in X\cap C(\R^N)$.
\end{lemma}
\begin{proof}
Since $w\in L^{\sigma}(\R^N)$ for all $\sigma\in [1,\infty]$, it belongs, in particular, to $L^{{p^*}'}(\R^N)$, which is continuously embedded into $X^*$. It is well known that the operator $T=-\Delta_p$ defines a bounded, continuous, strictly  monotone (note $1< p <N$),  and coercive operator from $X$ into its dual through
$$
\langle Tv,\varphi\rangle =\int_{\R^N}|\nabla v|^{p-2}\nabla v\,\nabla\varphi\,dx, \quad \forall \varphi\in X,
$$
where $\langle \cdot,\cdot\rangle$ denotes the duality pairing between $X$ and $X^*$.
Thus $T: X\to X^*$ is  bijective, which yields the existence of a unique solution $v$ of (\ref{303}), which is even $C^{1,\gamma}_{\mathrm{loc}}(\R^N)$-regular. Next, we show that $v(x)\ge 0$. As a weak solution $v$ satisfies
$$
\int_{\R^N}|\nabla v|^{p-2}\nabla v\,\nabla\varphi\,dx=\int_{\R^N}Cw\varphi\,dx.
$$
Testing this relation with $\varphi=v^-=\max\{-v,0\}$, we get
$$
0\le \int_{\R^N}|\nabla v|^{p-2}\nabla v\,\nabla v^-\,dx=-\int_{\R^N}|\nabla v^-|^p\,dx\le 0,
$$
which implies that $\|v^-\|_X=0$ and thus $v^-=0$, that is, $v(x)\ge 0$ for all $x\in \R^N$, and by Harnack's inequality it
follows that $v(x)>0$ for all $x\in \R^N$.
\end{proof}
Next we are going to study pointwise estimate of the positive solution $v$ of (\ref{303}). To this end we note that the solution $v$ can be regarded as a solution of the following equation
\begin{equation}\label{304}
v\in W^{1,p}_{\mathrm{loc}}(\R^N)\cap C(\R^N): \ -\Delta_pv=\mu\quad\mbox{ in }\R^N,
\end{equation}
where $\mu$ is the nonnegative Radon measure generated by the positive function $\hat w(x)= C w(x)$ through
\begin{equation}\label{305}
\mu(E)= \int_E\hat w(x) \, dx,\quad E\subset \R^N,\ \ E \mbox{ Lebesque measurable}.
\end{equation}
Therefore, the unique positive solution $v$ is a special case of what is usually referred to as $A$-superharmonic function, see  \cite{HKM93,Kilp94}. By results due to Kilpelainen-Maly (see \cite[Corollary 4.13]{Kilp94}), and in view of $\inf_{x\in \R^N} v(x)=0$
we get the following pointwise estimate.
\begin{lemma}\label{L302}
The unique positive solution $v$ satisfies the following pointwise estimate
\begin{equation}\label{306}
c_1 W_{1,p}^\mu(x, \infty)\le v(x)\le c_2 W_{1,p}^\mu(x, \infty),\quad x\in \R^N,
\end{equation}
where $c_1,c_2$ are positive constants depending on $N,p$, and $W_{1,p}^\mu(x, \infty)= \lim_{R\to\infty} W_{1,p}^\mu(x, R)$ with $W_{1,p}^\mu(x, R)$ being the Wolff potential from nonlinear potential theory defined by
\begin{equation}\label{307}
 W_{1,p}^\mu(x,R)= \int_0^R \Big(\frac{\mu(B(x,t))}{t^{N-p}}\Big)^{\frac{1}{p-1}}\frac{dt}{t}, \quad R>0,\  x\in \R^N.
\end{equation}
\end{lemma}
Thus a pointwise estimate of $v$ from above is provided by an estimate from above of the Wolff potential
$W_{1,p}^\mu(x, \infty)$, which has been calculated in \cite{Carl-17, Carl-18}. In particular, by \cite[Lemma 2.1, Theorem 2.2]{Carl-18} we obtain the following result.

\begin{lemma}\label{L303}
The unique positive solution $v\in X \cap C(\R^N)$ satisfies the following pointwise estimate:
\begin{equation}\label{308}
0<v(x)\le c\frac{1}{1+|x|^{\frac{N-p}{p-1}}},\quad x\in \R^N.
\end{equation}
\end{lemma}

The decay estimate of the positive solution $v$ will be used to get a similar decay estimate for any solution of the original problem (\ref{301}). We have the following result.
\begin{theorem}\label{T301}
Assume the hypotheses (Ha) and (Hg), and let $u\in X$ be a solution of (\ref{301}). Then $u\in X\cap C^{1,\gamma}_{\mathrm{loc}}(\R^N)$, $\gamma\in (0,1)$, satisfies
\begin{equation}\label{309}
|u(x)|\le C\frac{1}{1+|x|^{\frac{N-p}{p-1}}},\quad x\in \R^N,
\end{equation}
with the constant $C=C(N,p, |u|_{p^*})$.
\end{theorem}
\begin{proof}
By Corollary \ref{C301}, $u\in L^\infty(\R^N)$ and hence the right hand side is bounded with
$$
|a(x)g(x,u(x))|\le C w(x)= \hat w(x)
$$
Since $v$ is the unique positive solution of (\ref{303}) (see Lemma \ref{L301}) we get by comparison
\begin{equation}\label{310}
\langle -\Delta_p u- (-\Delta_pv), \varphi\rangle \le 0,\quad \forall \varphi\in X_+,
\end{equation}
where $X_+=\{\varphi\in X: \varphi\ge 0\}$. Taking in (\ref{310}) the test function $\varphi=(u-v)^+$ we get
\begin{eqnarray*}
0&\ge &\int_{\R^N} \Big(|\nabla u|^{p-2}\nabla u-|\nabla v|^{p-2}\nabla v\Big)\nabla (u-v)^+\,dx\\
&=&\int_{\{u>v\}} \Big(|\nabla u|^{p-2}\nabla u-|\nabla v|^{p-2}\nabla v\Big)(\nabla u-\nabla v)\,dx\ge 0,
\end{eqnarray*}
where $\{u>v\}:=\{x\in \R^N: u(x)> v(x)\}$. Since $A(x,\xi)=|\xi|^{p-2}\xi$ satisfies (A3), we infer that the Lebesgue measure of $\{u>v\}$ is zero,  and thus $(u-v)^+=0$, i.e., $u\le v$. Multiplying (\ref{301}) by $-1$ we get
$$
-(-\Delta_pu)=-a(x) g(x,u)
$$
that is
$$
\int_{R^N}|\nabla (-u)|^{p-2}\nabla (-u)\nabla \varphi\,dx=\int_{R^N}-a(x) g(x,u)\,dx.
$$
Setting $\tilde u=-u$ and making use of the solution $v$ we get by subtracting the corresponding equations
$$
\int_{R^N}\Big(|\nabla \tilde u|^{p-2}\nabla \tilde u-|\nabla v|^{p-2}\nabla v\Big)\nabla \varphi\,dx\le 0,\quad \forall \varphi\in X_+,
$$
which by applying the test function $\varphi=(\tilde u-v)^+$ yields $(\tilde u-v)^+$ and thus $\tilde u\le v$, or equivalently $u\ge -v$. Hence we have obtained $|u|\le v$, which by using the decay property of $v$ according to Lemma \ref{L303} completes the proof.
\end{proof}

\smallskip

Finally, we are going to provide sufficient conditions for the existence of solutions of equation (\ref{301}).

\begin{theorem}\label{T302}
Let $a: \R^N\to\R$ satisfy (Ha) and let $g:\R^N\times \R\to\R$ satisfy either
\begin{itemize}
\item[(Hg1)] $|g(x,s)|\le c_g\big(1+|s|^{q-1}\big), \quad\mbox{with } 1<q<p,$ and $c_g>0$,
\end{itemize}
or
\begin{itemize}
\item[(Hg2)] $|g(x,s)|\le c_g\big(1+|s|^{p-1}\big)$,  with  $c_g>0$ such that $c_ac_g<\frac{1}{\|i_w\|^p}$, where $\|i_w\|$ denotes the norm of the embedding operator $i_w: X\to L^q(\R^N,w)$ with $L^q(\R^N,w)$ being the weighted Lebesgue space.
\end{itemize}
Then (\ref{301}) admits at least one solution.
\end{theorem}

\begin{proof}
Let us assume hypotheses (Ha) and (Hg1). If $G$ denotes the Nemytskij operator generated by $g$ through $G(u)(x)=g(x,u(x))$, then $G: L^q(\R^N, w)\to L^{q'}(\R^N, w)$ is a bounded and continuous mapping from the weighted Lebesgue space $L^q(\R^N, w)$ with weight function $w$ into its dual space, where $w$ is given by (\ref{300}), and $L^q(\R^N, w)$ is defined by
$$
L^q(\R^N, w)=\Big\{u: \R^N\to \R \mbox{ measurable}: \int_{\R^N}w|u|^q\,dx <\infty\Big\},
$$
which is separable and reflexive under the norm
$$
|u|_{q,w}=\Big(\int_{\R^N}w|u|^q\,dx\Big)^{\frac{1}{q}}.
$$
In \cite[Lemma 6.1]{C-L-21} it is proved that the embedding $X\hookrightarrow\hookrightarrow L^q(\R^N, w)$ is compact. Thus, if $i_w: X\to  L^q(\R^N, w)$ denotes the embedding operator, then the composed operator
$$
G\circ i_w: X\to L^{q'}(\R^N, w) \quad\mbox{ is bounded and completely continuous}.
$$
Further, let us introduce the operator $i_a^*: L^{q'}(\R^N, w)\to X^*$, which is related with the coefficient $a$ and defined by
$$
\eta\in L^{q'}(\R^N, w): \langle i_a^*\eta, \varphi\rangle:=\int_{\R^N}a \eta \varphi\,dx, \quad\forall \varphi\in X.
$$
One readily verifies that $i_a^*: L^{q'}(\R^N, w)\to X^*$ is linear and bounded. The latter follows from the estimate
\begin{eqnarray*}
|\langle i_a^*\eta, \varphi\rangle|&\le & c_a\int_{\R^N}w |\eta| |\varphi|\,dx\le c_a\int_{\R^N}w^{\frac{1}{q'}} |\eta| w^{\frac{1}{q}} |\varphi|\,dx\\
&\le & c_a |\eta|_{q', w}|\varphi|_{q,w}\le C|\eta|_{q', w} \|\varphi\|_X,\quad\forall \varphi\in X.
\end{eqnarray*}
Hence it follows that the operator
$$
aG=i_a^*\circ G\circ i_w: X\to X^* \quad\mbox{is bounded and completely continuous}.
$$
As $-\Delta_p: X\to X^*$ is continuous, bounded and strictly monotone, the operator
$$
-\Delta_p-aG: X\to X^* \quad\mbox{is bounded, continuous and pseudomonotone in the sense of Brezis}.
$$
By the main theorem on pseudomonotone operators (see e.g. \cite[Theorem 27.A]{Zei90}) the operator $-\Delta_p-aG: X\to X^* $ is surjective provided this operator is coercive, which will be verified next  completing  the existence proof.

First, we note that in view of \cite[Lemma 6.1]{C-L-21}, the embedding $X\hookrightarrow\hookrightarrow L^{\sigma}(\R^N, w)$ is compact for all $\sigma$ with $1<\sigma< p^*$. Applying H\"older and Young's inequality we have (note: $1<q<p$)
\begin{eqnarray*}
\langle aG(u),u\rangle|&\le & \int_{\R^N}|a| |g(x,u| |u|\,dx\le c_ac_g \int_{\R^N} \big(w|u|+w(C(\varepsilon)+\varepsilon|u|^p)\big)\,dx\\
&\le& C |w|_{{p^*}'}|u|_{p*} +C(\varepsilon)|w|_1 +\varepsilon |u|_{p,w}^p\le C(\tilde\varepsilon)(1+\|u\|_X)+ \tilde\varepsilon \|u\|_X^p,
\end{eqnarray*}
where $\tilde\varepsilon>0$ can be arbitrarily chosen. Choosing $\tilde\varepsilon=\frac12$ we get
$$
\frac{1}{\|u\|_X}\langle -\Delta_p u-aG(u), u\rangle \ge \frac12\|u\|_X^{p-1}-C(\tilde\varepsilon)\big(1+\frac{1}{\|u\|_X}\big) \to\infty,\ \mbox{ as } \|u\|_X\to \infty,
$$
hence the coercivity, which completes the proof under the assumptions (Ha) and (Hg1).

The existence proof in case the assumption (Hg1) is replaced by (Hg2) follows basically the same line as before, where in this case the smallness of the constant $c_ac_g$ is needed to ensure  the coercivity of $-\Delta_p-aG: X\to X^* $. To be more precise let us check the coercivity. Denote the norm of the embedding operator $i_w: X\to L^p(\R^N, w)$ by $\|i_w\|$, and the norm of the embedding $i: X\to L^{p*}(\R^N)$ by $\|i\|$. Then we get the following estimate
\begin{eqnarray*}
\langle -\Delta_p u-aG(u), u\rangle &\ge& \|u\|_X^p-c_ac_g|w|_{{p^*}'}|u|_{p*}-c_ac_g |u|_{p,w}^p\\
&\ge & \big(1-c_ac_g \|i_w\|^p\big) \|u\|_X^p-\big(c_ac_g|w|_{{p^*}'} \|i\|\big) \|u\|_X.
\end{eqnarray*}
From the last inequality we infer coercivity provided  $ 1-c_ac_g \|i_w\|^p > 0$, i.e., $c_ac_g<\frac{1}{\|i_w\|^p}$, which completes the proof.
\end{proof}
%

\section*{Acknowledgment}

We are very grateful for the reviewer's careful reading of the manuscript and helpful comments to improve its content and readability.


\section*{Declarations}

\noindent{\bf Funding:} No funding was received to assist with the preparation of this manuscript.

\smallskip

\noindent{\bf Conflict of interest:} The authors have no conflicts of interest to declare that are relevant to the content of this article.

\noindent{\bf Data availability:} The authors declare that the data supporting the findings of this study are available within the paper.


\end{document}